\documentclass{amsart}

\usepackage{amsmath,amstext,amssymb,mathrsfs,amscd,amsthm,indentfirst}
\usepackage{amsfonts}
\usepackage{xypic}
\usepackage{enumerate}
\usepackage{verbatim}
\usepackage{pstricks}
\usepackage{graphicx}

\usepackage{hyperref}

\theoremstyle{plain}
\newtheorem{theorem}{Theorem}[section]
\newtheorem{proposition}[theorem]{Proposition}
\newtheorem{lemma}[theorem]{Lemma}

\newtheorem{MainThm}{Theorem}

\theoremstyle{definition}
\newtheorem{definition}[theorem]{Definition}

\newtheorem{remark}[theorem]{Remark}

\numberwithin{equation}{section}

\newcommand{\bQ}{\mathbb{Q}}
\newcommand{\bR}{\mathbb{R}}

\newcommand{\bF}{\mathbb{F}}

\newcommand{\bH}{\mathbb{H}}
\newcommand{\bP}{\mathbb{P}}

\newcommand{\cA}{\mathcal{A}}
\newcommand{\cl}{\mathcal{L}}

\newcommand{\ft}{\mathfrak{t}}

\newcommand{\sign}{\operatorname{sign}}

\newcommand{\Diff}{\mathrm{Diff}}
\newcommand{\Homeo}{\mathrm{Homeo}}

\newcommand{\cha}{\mathrm{char}}

\newcommand{\id}{\mathrm{id}}

\newcommand{\blockdiff}{\widetilde{\Diff}}

\newcommand{\hofib}{\operatorname{hofib}}
\newcommand{\opstar}{\mathrm{St}^\circ}
\renewcommand{\star}{\mathrm{St}}

\newcommand{\cM}{\mathcal{M}}

\newcommand{\inter}{\operatorname{int}}
\newcommand{\Link}{\operatorname{Lk}}

\newcommand{\trf}{\operatorname{trf}}

\newcommand{\Gr}{\operatorname{Gr}}

\newcommand{\Sq}{\operatorname{Sq}}
\newcommand{\thom}{\operatorname{th}}

\newcommand\lra{\longrightarrow}

\title[Generalised Miller--Morita--Mumford classes]{Generalised Miller--Morita--Mumford classes for block bundles and topological bundles}

\author{Johannes Ebert}
\thanks{}
\email{johannes.ebert@uni-muenster.de}
\address{Mathematisches Institut der
Westf{\"a}lische Wilhelms-Universit{\"a}t M{\"u}nster\\
Einsteinstr. 62\\
DE-48149 M{\"u}nster\\
Germany}

\author{Oscar Randal-Williams}
\thanks{}
\email{o.randal-williams@dpmms.cam.ac.uk}
\address{Centre for Mathematical Sciences\\
Wilberforce Road\\
Cambridge CB3 0WB\\
UK}



\begin{document}

\begin{abstract}
The most basic characteristic classes of smooth fibre bundles are the generalised Miller--Morita--Mumford classes, obtained by fibre integrating characteristic classes of the vertical tangent bundle. In this note we show that they may be defined for more general families of manifolds than smooth fibre bundles: smooth block bundles and topological fibre bundles.
\end{abstract}

\maketitle

\section{Introduction}

Let $M$ be a smooth, closed, oriented manifold of dimension $d$, and $\pi : E \to B$ be a fibre bundle with fibre $M$ and structure group $\Diff^+(M)$, the topological group of orientation-preserving diffeomorphisms. The \emph{vertical tangent bundle} is a $d$-dimensional oriented vector bundle $T_v E \to E$, which may be constructed from the principal $\Diff^+(M)$-bundle $P \to B$ associated to $\pi$ by
$$T_v E := P \times_{\Diff^+(M)} TM,$$
using the action of $\Diff^+(M)$ on $TM$ which takes the differential of a diffeomorphism.

Recall that the ring of characteristic classes of oriented $d$-dimensional vector bundles with coefficients in the field $\bF$ is $H^*(BSO(d);\bF)$. If $\cha(\bF) \neq 2$ we have
\begin{align*}
H^* (BSO(2m); \bF) &\cong \bF[e,p_1,\ldots,p_m]/(e^2-p_m)\\
H^* (BSO(2m+1); \bF) &\cong \bF[p_1,\ldots,p_m],
\end{align*}
where $p_i$ is the $i$th Pontrjagin class and $e$ is the Euler class, while for $\cha(\bF) = 2$ we have
$$H^* (BSO(n); \bF) \cong \bF [w_2,\ldots, w_n],$$
where $w_i$ is the $i$th Stiefel--Whitney class. Thus for any monomial $c$ of degree $k$ in these classes, we may evaluate the characteristic class $c(T_v E) \in H^k(E;\bF)$, then push it forwards along the map $\pi$ to obtain
$$\kappa_c(\pi) := \pi_!(c(T_v E)) \in H^{k-d}(B;\bF),$$
the generalised Miller--Morita--Mumford class, hereafter MMM-class, associated to $c$. This construction tautologically yields characteristic classes for smooth oriented fibre bundles with $d$-dimensional fibres. In particular, there are universal classes
$$\kappa_c \in H^{k-d}(B\Diff^+(M);\bF).$$

In this note we investigate to what extent the characteristic classes $\kappa_c$ may be defined on more general families of manifolds: block bundles and topological bundles. (We will recall the notion of a block bundle in Definition \ref{defn:block-bundle}.) 

By the relation $e^2 = p_m$ in the cohomology of $BSO(2m)$ with coefficients in a field $\bF$ of characteristic not 2, we may write any monomial in $H^*(BSO(2m);\bF)$ in the form $e^\epsilon \cdot p_1^{i_1} \cdots p_m^{i_m}$ with $\epsilon = 0$ or $1$. Let us then define $\bF[p_1, p_2, \ldots] \langle 1, e \rangle$ to be the vector space over $\bF$ with basis the monomials in $e$ and the $p_i$ where $e$ occurs with exponent at most 1. 

\begin{MainThm}\label{thm:A}
Let $\bF$ be any field, and fix a dimension $d$. If $\cha(\bF) \neq 2$ then let $c$ be a monomial in $\bF[p_1, p_2, \ldots]$ if $d$ is odd, or a monomial in $\bF[p_1, p_2, \ldots] \langle 1, e \rangle$ if $d$ is even. If $\cha(\bF)=2$ then let $c$ be a monomial in $\bF[w_1, w_2, \ldots]$. There is defined for each oriented smooth block bundle $(p : E \to \vert K\vert, \mathcal{A})$ with $d$-dimensional fibres over a simplicial complex a class
$$\tilde{\kappa}_{c}(p, \mathcal{A}) \in H^{*}(\vert K\vert;\bF)$$
such that
\begin{enumerate}[(i)]
\item If $f : L \to K$ is a simplicial map, and $(f^*p : f^*E \to \vert L\vert, f^*\mathcal{A})$ is the pull-back block bundle, then $f^* \circ \tilde{\kappa}_{c}(p, \mathcal{A}) = \tilde{\kappa}_{c}(f^*p, f^*\mathcal{A})$.
\item If the block bundle $(p : E \to \vert K\vert, \mathcal{A})$ arises from a smooth fibre bundle $\pi$, then $\tilde{\kappa}_{c}(p, \mathcal{A}) = \kappa_{c}(\pi)$.
\end{enumerate}
\end{MainThm}

We have a similar statement for topological bundles, but only when the coefficients are a field of characteristic zero or two.

\begin{MainThm}\label{thm:B}
Let $\bF$ be a field of even characteristic, and fix a dimension $d$. If $\cha(\bF) = 0$ then let $c$ be a monomial in $\bF[p_1, p_2, \ldots]$ if $d$ is odd, or a monomial in $\bF[e, p_1, p_2, \ldots]$ if $d$ is even. If $\cha(\bF)=2$ then let $c$ be a monomial in $\bF[w_1, w_2, \ldots, w_d]$. For each fibre bundle $\pi : E \to B$ with fibre a $d$-dimensional oriented topological manifold $M$ and structure group the orientation-preserving homeomorphisms of $M$, there is defined a class
\begin{align*}
{\kappa}^{TOP}_{c}(\pi) \in H^*(B ; \bF)
\end{align*}
such that
\begin{enumerate}[(i)]
\item If $f : C \to B$ is a continuous map, and $f^*\pi : f^*E \to C$ is the pull-back bundle, then $f^* \circ {\kappa}^{TOP}_{c}(\pi) = {\kappa}^{TOP}_{c}(f^*\pi)$.
\item If $M$ admits a smooth structure and the bundle $\pi : E \to K$ admits a reduction of its structure group to the orientation-preserving diffeomorphisms, then ${\kappa}^{TOP}_{c}(\pi) = \kappa_{c}(\pi)$.
\end{enumerate}
\end{MainThm}

For both block bundles and topological bundles, we will show that the characteristic classes are in fact defined on suitable classifying spaces (Theorem \ref{thm:Aprime} and Proposition \ref{prop4.2}).

On smooth bundles the characteristic classes $\kappa_c$ satisfy the following obvious property: if the monomial $c$ contains a Pontrjagin class $p_i$ with $2i > d$, then $\kappa_c$ vanishes on any smooth bundle $\pi : E \to B$ with $d$-dimensional fibres. This is simply because $p_i(T_v E)$ already vanishes under these conditions. This property is by no means clear for topological bundles: we refer the reader to \cite{ReisWeiss} for a recent detailed discussion of this question. For block bundles however, the analogue of this property fails: we provide a counterexample.

\begin{MainThm}\label{thm:counterexample}
There exists a smooth oriented block bundle over a simplicial complex homeomorphic to $S^{12}$ with fibres homotopy equivalent to $\bH \bP^2$, having $\tilde{\kappa}_{p_5} \neq 0$.
\end{MainThm}

The existence of generalised MMM-classes for topological bundles and smooth block bundles may be used to prove the following theorem, comparing smooth automorphisms of certain basic manifolds with their continuous or smooth block automorphisms. Let us write $W^{2n}_g := \#^g S^n \times S^n$ for the connected-sum of $g$ copies of $S^n \times S^n$, and consider homeomorphisms and diffeomorphisms of $W_g^{2n}$ relative to a fixed disc $D^{2n} \subset W^{2n}_g$. 

\begin{MainThm}\label{thm:app}
The maps
$$B\Diff(W_{g}^{2n}, D^{2n}) \lra B\Homeo(W_{g}^{2n}, D^{2n})$$
and
$$B\Diff(W_{g}^{2n}, D^{2n}) \lra B\blockdiff(W_{g}^{2n}, D^{2n})$$
are surjective on rational cohomology in degrees $* \leq \tfrac{g-4}{2}$, and injective on rational cohomology in degrees $* \leq \min(\tfrac{2n-7}{2}, \tfrac{2n-4}{3})$. In this range of degrees the common rational cohomology is the polynomial algebra generated by generalised MMM-classes described in \cite[Theorem 1.1]{GRW2}.
\end{MainThm}

\subsection{Acknowledgements}

We would like to thank Wolfgang Steimle for introducing us to Casson's work, which plays a central role in the proof of Theorem \ref{thm:counterexample}. O.\ Randal-Williams was supported by the Herchel Smith Fund.

\section{Block bundles}\label{section:block-bundles}

We will first review some basic notions from the theory of block bundles, essentially from \cite{RSII}. Throughout we will use the term \emph{semi-simplicial set} to mean a ``simplicial set without degeneracies", i.e.\ a $\Delta$-set in the sense of \cite{RSI}.

Suppose that $M$ is a smooth, closed manifold. The basic example of a block bundle with fibre $M$ over a simplex $\Delta^p$ is a map
\begin{equation}\label{eq:BasicBlockBundle}
p : \Delta^p \times M \lra \Delta^p
\end{equation}
such that for each face $\sigma \subset \Delta^p$ the map $\pi$ sends $\sigma \times M$ to $\sigma$. Roughly speaking, a block bundle over $\vert K \vert$, the geometric realisation of a simplicial complex, is a map assembled by gluing together maps of this form over the simplices of $K$. The natural form of gluing to allow is that of a block diffeomorphism.

\begin{definition}\label{defn:blockdiffs}
A \emph{block diffeomorphism} of $\Delta^p \times M$ is a diffeomorphism
$$f : \Delta^p \times M \lra \Delta^p \times M$$
which for each face $\sigma \subset \Delta^p$ restricts to a diffeomorphism of $\sigma \times M$.

The semi-simplicial group of {block diffeomorphisms} $\blockdiff(M)_\bullet$ has as its group of $p$-simplices the set of all block diffeomorphisms of $\Delta^p \times M$. The group operation is by composition, and the semi-simplicial structure is given by restriction to the faces. Similarly, if $M$ is oriented we may define $\blockdiff^+(M)_\bullet$ using only the orientation-preserving diffeomorphisms of $\Delta^p \times M$.
\end{definition}

The absolutely key point of this definition is that we do \emph{not} require that the diffeomorphism $f$ commutes with projection to $\Delta^p$: this difference will distinguish diffeomorphisms and block diffeomorphisms.

The topological group of diffeomorphisms of $M$, $\Diff(M)$, equipped with the Whitney $C^\infty$-topology, is contained in the block diffeomorphism group in the following manner. Call a continuous map $\sigma: \Delta^p  \to \Diff(M)$ \emph{smooth} if the induced homeomorphism $(t, x) \mapsto (t, \sigma(t)\cdot x) : \Delta^p \times M \to \Delta^p \times M$ is a diffeomorphism. This defines a semi-simplicial subgroup $\mathrm{Sing}_{\bullet}^{sm} \Diff (M) \subset \mathrm{Sing}_{\bullet} \Diff (M)$ of the semi-simplicial group of singular simplices, and the inclusion is a homotopy equivalence.

The diffeomorphism of $\Delta^p \times M$ induced by an element of $\mathrm{Sing}_{p}^{sm} \Diff (M)$ preserves the face structure (and, in fact, the projection to $\Delta^p$). This observation determines an inclusion of semi-simplicial groups 
$$\mathrm{Sing}_{\bullet}^{sm} \Diff (M) \hookrightarrow \blockdiff(M)_{\bullet}.$$
The classifying space $B\blockdiff(M)$ is by definition the geometric realisation of the bi-semi-simplicial set $N_{\bullet} \blockdiff(M)_{\bullet}$ obtained by taking the semi-simplicial nerve levelwise. There are maps
$$B\Diff(M) \overset{\sim}\longleftarrow \vert N_\bullet \mathrm{Sing}_{\bullet}^{sm} \Diff (M)\vert \lra \vert N_{\bullet} \blockdiff(M)_{\bullet} \vert = B\blockdiff(M),$$
where the leftwards map is induced by the evaluation
$$(N_p \mathrm{Sing}_{q}^{sm} \Diff (M)) \times \Delta^q \lra N_p \Diff (M),$$
and is a weak homotopy equivalence. We will always use these maps to compare ordinary and block diffeomorphisms.

Let us describe another model for $B\blockdiff(M)$, which is a semi-simplicial (as opposed to bi-semi-simplicial) set and has the advantage of being more geometric in flavour. We will then prove that the two models are homotopy equivalent. 

\begin{definition}\label{def-blockbundlespace}
Let $\mathcal{M}(M ; \bR^N)_{p}$ be the set of submanifolds $W \subset \Delta^p \times \bR^N$ which are transverse to $\sigma \times \bR^N$ for each face $\sigma \subset \Delta^p$, and which are diffeomorphic to $\Delta^p \times M$ via a diffeomorphism taking $W \cap (\sigma \times \bR^N)$ to $\sigma \times M$. Define face maps $d_i : \mathcal{M}(M ; \bR^N)_{p} \to \mathcal{M}(M ; \bR^N)_{p-1}$ by intersecting $W$ with the $i$th face of $\Delta^p$; this gives a semi-simplicial set $\mathcal{M}(M ; \bR^N)_{\bullet}$. Finally, let $\mathcal{M}(M)_{\bullet} := \cup_N \mathcal{M}(M ; \bR^N)_{\bullet}$.

Similarly, if $M$ is oriented define $\mathcal{M}(M)_{\bullet}^+$ using oriented manifolds $W$ and orientation-preserving diffeomorphisms to $\Delta^p \times M$.
\end{definition}

\begin{proposition}
There is a weak homotopy equivalence $\vert\cM(M)_{\bullet}\vert \simeq \vert N_{\bullet} \blockdiff(N)_{\bullet}\vert$. Similarly, if $M$ is oriented there is a weak homotopy equivalence $\vert\cM^+(M)_{\bullet}\vert \simeq \vert N_{\bullet} \blockdiff^+(N)_{\bullet}\vert$.
\end{proposition}

\begin{proof}
We will prove the first statement, as the second is a minor modification. We introduce an auxiliary bi-semi-simplicial set $X_{\bullet,\bullet}$. The set $X_{p,q}$ of $(p,q)$-simplices consists of tuples $(W,f_0,\ldots,f_p)$, where $W\in \cM(M)_q$ and $f_i:W \to \Delta^q \times M$ is a diffeomorphism as in Definition \ref{def-blockbundlespace}. The $i$th face map in the $q$ direction is given by intersecting $W$ with the $i$th face of $\Delta^q$ and then restricting the $f_j$, and the $i$th face maps in the $p$ direction is by forgetting $f_i$. There is an augmentation $F_{\bullet, q}:X_{\bullet,q} \to \cM(M)_q$ by forgetting all the $f_i$.

The fibre of $F$ over each $q$-simplex has contractible geometric realisation, by general nonsense. Namely, for any nonempty set $Y$, the semi-simplicial set $Y_{\bullet}$ with $Y_p=Y^{p+1}$ and the forgetful maps as simplicial structure maps, is contractible. The preimage $F^{-1}(W)$ is just this construction, applied to the set of diffeomorphisms $W \cong \Delta^q \times M$, which is required to be nonempty by definition. Thus the map
$$\vert F_{\bullet, q}\vert : \vert X_{\bullet, q} \vert \lra  \cM(M)_q$$
is a homotopy equivalence, and so
$$\vert F_{\bullet, \bullet}\vert : \vert X_{\bullet, \bullet}\vert \lra \vert \cM(M)_\bullet \vert$$
is too.

There is a map  of bi-semi-simplicial sets $G_{p,q}: X_{p,q}\to N_p \blockdiff (M)_q$ given by 
$$(W,f_0,\ldots,f_p) \mapsto (f_{0} f^{-1}_{1} , f_{1} f^{-1}_{2}, \ldots , f_{p-1} f^{-1}_{p}).$$
For fixed $p$, the semi-simplicial map $G_{p,\bullet}: X_{p,\bullet} \to N_{p} \blockdiff(M)_{\bullet}$ is a Kan fibration. Thus the homotopy fibre after geometric realisation may be computed simplicially. In the context of semi-simplicial sets, this means the following. The target of $G_{p,\bullet}$ may be made into a pointed semi-simplicial set (i.e.\ a semi-simplicial object in pointed sets) by choosing a ``basepoint'' $x_{\bullet}$, which is a sub-semi-simplicial set with a single element in each degree. Let $x:=(h_1,\ldots,h_p) \in N_p \blockdiff{M}_0$ be a $0$-simplex, so the $h_i$ are diffeomorphisms of $M$, and let $x_{q}$ be the $q$-simplex $(h_1 \times \id_{\Delta^q},\ldots,h_p \times \id_{\Delta^q}) \in N_p \blockdiff{M}_q$. We claim that $G_{p,\bullet}^{-1}(x_{\bullet}) \subset X_{p,\bullet}$ is contractible. It is the semi-simplicial set with $q$-simplices those $(W, f_0, \ldots, f_p)$ with $f_{i-1} f_i^{-1}=h_i \times \id_{\Delta^q}$. Since $h_i$ is fixed, $f_0$ determines all $f_i$ uniquely.
So $G_{p,\bullet}^{-1}(x_{\bullet})$ is isomorphic to the semi-simplicial set $A_{\bullet}$, where $A_q$ is the set of all $(W,f)$, $W \in \cM(M)_q$ and $f: W \to \Delta^q \times M$ is a diffeomorphism as in Definition \ref{def-blockbundlespace}. So $|A_{\bullet}|$ is the block embedding space of $M$ in $\bR^\infty$, which is contractible by Whitney's embedding theorem.
\end{proof}

The structure that is classified by the space $B \blockdiff (M) := \vert N_\bullet \blockdiff(M)_\bullet \vert \simeq \vert \mathcal{M}(M)_\bullet \vert$ is that of a {block bundle}. We define this notion only when the base is a simplicial complex, and for simplicity only when $M$ is a closed manifold.

\begin{definition}\label{defn:block-bundle}
Let $K$ be a simplicial complex and $p: E \to \vert K\vert$ be a continuous map. A \emph{block chart} for $E$ over a simplex $\sigma \subset K$ is a homeomorphism $h_\sigma: p^{-1} (\sigma) \to \sigma \times M$ which for every face $\tau \subset \sigma$ restricts to a homeomorphism $p^{-1}(\tau) \to \tau \times M$. A \emph{block atlas} is a set $\cA$ of block charts, at least one over each simplex of $K$, such that if $h_{\sigma_i}: p^{-1}(\sigma_i) \to \sigma_i \times M$, $i=0,1$, are two elements of $\cA$ then the composition $h_{\sigma_1} \circ h_{\sigma_0}^{-1}$ from $(\sigma_0 \cap \sigma_1) \times M$ to itself is a block diffeomorphism. A \emph{block bundle structure} is a maximal block atlas.  The resulting structure is a \emph{block bundle}.
\end{definition}

Suppose that $\pi: E \to \vert K \vert$ is a smooth fibre bundle with fibre $M$, and hence locally trivial. As simplices are contractible and paracompact, the restriction of $\pi$ to each simplex is a trivial bundle, and for each simplex $\sigma$ we may choose trivialisations $\pi^{-1}(\sigma) \overset{\sim}\to \sigma \times M$ over $\sigma$. These trivialisations provide a block atlas for $\pi$, exhibiting it as a block bundle. Hence every smooth fibre bundle over a simplicial complex yields a smooth block bundle.

If $(p : E \to \vert K \vert, \mathcal{A})$ is a block bundle, and $f: L \to K$ is a map of simplicial complexes, we define the pull-back block bundle to have as representing space the projection map $q: E \times_{\vert K \vert} \vert L \vert \to \vert L \vert$. The surjective simplicial map $f\vert_\tau : \tau \to f(\tau)$ expresses $\tau$ as the join $*_{v \in V}X_{v}$ of simplices $X_v = f^{-1}(v)$ indexed over the set $V$ of vertices of $f(\tau)$. For each block chart $h_{f(\tau)}$ we define a block chart over $\tau$ by
\begin{align*}
h_\tau : q^{-1}(\tau) = p^{-1}(f(\tau)) \times_{f(\tau)} \tau &\lra \tau \times M\\
\left(e, \sum_{v \in V} t_v \cdot x_v\right) & \longmapsto \left(\sum_{v \in V} t'_v \cdot x_v, \pi_M(h_{f(\tau)}(e))\right)
\end{align*}
where the $t'_v$ are defined by $\pi_{f(\tau)}(h_{f(\tau)}(e)) = \sum_{v \in V} t'_v \cdot v \in f(\tau)$. Note that there exists an $h_{f(\tau)}$ for each $\tau$, as $\mathcal{A}$ is assumed to be a maximal atlas and so contains a block chart for every simplex.

\begin{lemma}
The functions $h_\tau$ are well-defined homeomorphisms, and the transition maps $h_\tau \circ h_\sigma^{-1}$ are block diffeomorphisms of $(\tau \cap \sigma) \times M$.
\end{lemma}
\begin{proof}
To see $h_\tau$ is well-defined note that the only ambiguity is that if $t_v=0$ then $x_v$ is undefined. But in this case $p(e) = f\circ q(e, \sum t_v \cdot x_v) = \sum t_v \cdot v \in f(\tau)$ lies in the face opposite $v$, so $h_{f(\tau)}(e)$ also lies in the face opposite $v$, so $t'_v=0$ too. The function $h_\tau$ is clearly continuous, and a continuous inverse is given by the formula
\begin{align*}
h_\tau^{-1} : \tau \times M  &\lra q^{-1}(\tau) = p^{-1}(f(\tau)) \times_{f(\tau)} \tau\\
\left(\sum_{v \in V} t'_v \cdot x_v, m\right) & \longmapsto \left( h_{f(\tau)}^{-1}\left(\sum_{v \in V} t'_v \cdot v, m\right), \sum_{v \in V} t_v \cdot x_v\right)
\end{align*}
where the $t_v$ are defined by  $p(h_{f(\tau)}^{-1}(\sum_{v \in V} t'_v \cdot v, m)) = \sum_{v \in V} t_v \cdot v \in f(\tau)$.

If $f(\sigma \cap \tau)$ has vertices $V$, so that $\sigma \cap \tau = *_{v \in V} X_v$, then for $(\sum_{v \in V} t_v \cdot x_v, m) \in (\sigma \cap \tau) \times M$ we have
$$h_\tau \circ h^{-1}_\sigma \left(\sum_{v \in V} t_v \cdot x_v, m\right) = \left(\sum_{v \in V} \hat{t}_v \cdot x_v, h_{f(\tau)} \circ h_{f(\sigma)}^{-1}\left(\sum_{v \in V} t_v \cdot v, m\right)\right)$$
where $\pi_{f(\sigma \cap \tau)}(h_{f(\tau)} \circ h_{f(\sigma)}^{-1}(\sum t_v \cdot v, m)) = \sum \hat{t}_v \cdot v$. This is clearly smooth and preserves faces ($\hat{t}_v = 0$ if and only if $t_v=0$, as $h_{f(\tau)} \circ h_{f(\sigma)}^{-1}$ preserves faces), and $h_\sigma \circ h_\tau^{-1}$ is an inverse, so it is a block diffeomorphism.
\end{proof}

If the simplicial map $f: L \to K$ is simplexwise injective, there is a simpler, but equivalent, description of the pull-back block bundle. We again take the representing space to be $q: E \times_{\vert K \vert} \vert L \vert \to \vert L \vert$, equipped with block charts
$$q^{-1}(\tau) \approx p^{-1}(f(\tau)) \overset{h_{f(\tau)}}\lra f(\tau) \times M \approx \tau \times M$$
for each simplex $\tau$ of $L$ and each $h_{f(\tau)} \in \mathcal{A}$, where the two unlabelled homeomorphisms are induced by the homeomorphism $f\vert_\tau : \tau \to f(\tau)$. 

\begin{definition}\label{defn:concordance}
If $K$ is a simplicial complex, and $L$ is a triangulation of $\vert K \vert \times [0,1]$ which restricts to the triangulation $K$ at each end of the cylinder, then a block bundle $(p : E \to \vert L \vert, \mathcal{A})$ is called \emph{concordance} of block bundles on $\vert K \vert$. We say that the two block bundles on $\vert K \vert$ obtained by restricting $(p : E \to \vert L \vert, \mathcal{A})$ to $\vert K \vert \times \{0\}$ and $\vert K \vert \times \{1\}$ are \emph{concordant}.
\end{definition}

The notion of concordance defines a relation on the set of block bundles over $\vert K \vert$. It is clearly transitive and symmetric, and by choosing an ordering of $K$ so that we can form the cartesian product $K \times \Delta^1$, and pulling back $(p, \mathcal{A})$ along the projection $K \times \Delta^1 \to K$, we see that the concordance relation is reflexive.



\begin{proposition}\label{prop:classification}
If $K$ is a finite simplicial complex, the set of concordance classes of block bundles (with fibre $M$) is in bijection with the set of homotopy classes $[\vert K\vert,|\cM(M)_{\bullet}|]$.
\end{proposition}

\begin{proof}
We may choose an ordering of the vertices of $K$, giving a semi-simplicial set $K_\bullet$ with homeomorphic geometric realisation (as we have supposed that $K$ is finite). The semi-simplicial set $\cM(M)_{\bullet}$ is easily seen to be Kan (this is a consequence of the Whitney embedding theorem), so by Rourke and Sanderson's simplicial approximation theorem \cite[Theorem 5.3]{RSI}, each map $f : \vert K\vert \to |\cM(M)_{\bullet}|$ is homotopic to the geometric realisation of a semi-simplicial map $f_\bullet : K_\bullet \to \cM(M)_{\bullet}$. Then we let
$$E := \cup_{\sigma \subset K} f_\bullet(\sigma) \subset \vert K \vert \times \bR^\infty,$$
which has a canonical structure of a block bundle over $\vert K \vert$. For uniqueness, we use the relative version of simplicial approximation and $\vert K \times \Delta^1 \vert \approx \vert K \vert \times [0,1]$ to produce a concordance of block bundles.

In the converse direction, let $(p: E \to \vert K \vert, \mathcal{A})$ be a block bundle over a finite simplicial complex. By induction on the simplices of $K$, using the Whitney embedding theorem we may find a topological embedding $e : E \to \vert K \vert \times \bR^N$ for some $N \gg 0$ which on each block chart gives a smooth embedding $e \circ h_\sigma^{-1} : \sigma \times M \hookrightarrow \sigma \times \bR^N$. We then choose an ordering of the vertices of $K$ to obtain $K_\bullet$, and define a semi-simplicial map $K_\bullet \to \mathcal{M}(M;\bR^N)_\bullet \subset \mathcal{M}(M)_\bullet$ sending a simplex $\sigma : \Delta^p \hookrightarrow \vert K \vert$ to the manifold $\sigma^*e(p^{-1}(\sigma(\Delta^p))) \subset \Delta^p \times \bR^N$. A relative version of the Whitney embedding theorem shows that the homotopy class obtained is independent of the choice of embedding $e$, and furthermore depends only on the concordance class of $(p: E \to \vert K \vert, \mathcal{A})$.
\end{proof}

\begin{proposition}\label{block-bundles:quasifibration}
A block bundle $p:E \to |K|$ is a ``weak quasifibration'' in the sense that for each vertex $v \in K$, the comparison map $p^{-1} (v) \to \hofib_p (v)$ is a weak homotopy equivalence.
\end{proposition}
\begin{proof}
Let $p: E \to |K|$ be a block bundle and recall that $|K|$ is (the geometric realisation of) a simplicial complex $K$. We wish to apply the gluing theorem for quasifibrations by Dold and Thom, \cite[Satz 2.2]{DT}, but without further modification, it cannot be applied since it is not true (and not claimed) that all point-preimages of $p$ have the same weak homotopy type, only those over vertices of $|K|$.

To get around this problem, we use a seemingly arcane construction due to McCord \cite{McCord}. Namely, let $X_K$ be the quotient of $|K|$ that is obtained by collapsing all \emph{open} simplices to points ($X_K$ has one point for each simplex of $K$ and is of course not a Hausdorff space). Let $f:|K| \to X_K$ be the quotient map. McCord proved that $f$ is a weak homotopy equivalence, by showing that $f$ satisfies the assumptions of the gluing theorem for quasifibrations (and the point-preimages are open simplices, hence contractible). Let $p: E \to |K|$ be the block bundle under consideration. We will prove that the composition $f \circ p: E \to X_K$ is a quasifibration in the sense of Dold--Thom. Once this is done, the argument is finished as follows: if $v \in K$ is a vertex, then the diagram
\begin{equation*}
\begin{gathered}
\xymatrix{
p^{-1} (v) \ar[d] \ar[r] & \hofib_{p} (v) \ar[d] \\
(f \circ p)^{-1} (f(v)) \ar[r] & \hofib_{f\circ p} (f(v))
}
\end{gathered}
\end{equation*}
commutes. The left vertical map is a homeomorphism; the bottom map is a weak equivalence because $f \circ p$ is a quasifibration and the right vertical is a weak equivalence by McCord's theorem. So the top map is a weak homotopy equivalence, as asserted.

For a simplex $\sigma$ of $K$, we denote the open star by $\opstar (\sigma)\subset |K|$. Observe that $\opstar (\sigma) \cap \opstar (\tau)$ is nonempty iff $\sigma \cup \tau$ (as a set of vertices of $K$) is a simplex, in which case $\opstar (\sigma) \cap \opstar (\tau)= \opstar(\sigma \cup \tau))$.
The images $U_{\sigma}:=f(\opstar (\sigma))$ form an open covering of the quotient $X_K$, and this cover is closed under taking finite intersections. Now we claim that the sets $U_{\sigma}$ are distinguished (``ausgezeichnet'') in the sense of Dold--Thom. As the set $U_{\sigma}$ is contractible, this amounts to showing that for each $x \in U_{\sigma}$, the inclusion $(f \circ p)^{-1} (x) \to (f \circ p)^{-1} (U_{\sigma})$ is a weak homotopy equivalence. Let $\tau$ be the closure of $f^{-1}(x)$; this is a simplex of $K$, lying in the star of $\sigma$ and not in the link. So what we have to prove is that for each $\tau \subset \star (\sigma)$, $\tau \not \subset \Link_K (\sigma)$, the inclusion $p^{-1} (\inter \tau) \to p^{-1}(\opstar \sigma)$ is a weak equivalence. 

By the definition of a block bundle, for every simplex $\tau$ there is a homeomorphism $p^{-1}(\tau ) \cong \tau \times M$ restricting to a similar homeomorphism on each face of $\tau$. Thus, for each vertex $v$ of $\tau$, the inclusions $p^{-1}(v)\to p^{-1}(\tau) \leftarrow p^{-1}(\inter \tau)$ are homotopy equivalences. If we can show that for each vertex $w \in \sigma$, the inclusions  $p^{-1}(w)\to p^{-1}(\star (\sigma)) \leftarrow p^{-1}(\opstar(\sigma))$ are homotopy equivalences, then we can pick a common vertex $v$ of $\sigma$ and $\tau$ and observe that the diagram
$$
 \xymatrix{
 p^{-1}(\inter (\tau)) \ar[r] \ar[d] & p^{-1} (\opstar(\sigma)) \ar[d] \\
 p^{-1} (\tau) \ar[r] & p^{-1} (\star(\sigma))\\
  p^{-1} (v)\ar[r]^{\id} \ar[u] & p^{-1} (v) \ar[u]
 }$$
commutes and all vertical maps are weak equivalences, which finishes the argument.

For a subcomplex $X \subset K$ say a homeomorphism $p^{-1}(X) \cong X \times M$ is \emph{block smooth} if for each simplex $\tau \subset X$ it restricts to a homeomorphism $p^{-1}(\tau) \cong \tau \times M$, and this homeomorphism lies in the block atlas $\mathcal{A}$. Let us write $\Link(\sigma)$ for the link of the simplex $\sigma$, and recall that $\star(\sigma) = \sigma * \Link(\sigma)$. We claim that we may find a block smooth homeomorphism  $p^{-1} (\star(\sigma)) \cong \star (\sigma) \times M$, which will finish the proof of the proposition as then the maps we are trying to show are equivalences may be identified with
$$\{w\} \times M \lra \star(\sigma) \times M \longleftarrow \opstar(\sigma) \times M,$$
which are clearly homotopy equivalences. In order to do so, we choose indiscriminately a block diffeomorphism $h_\tau: p^{-1}(\tau) \cong \tau \times M$ for each simplex $\tau \subset \star(\sigma)$. The block diffeomorphism $h_\sigma$ provides a block smooth homeomorphism over $\sigma = \sigma * \emptyset$, and we extend this to a block smooth homeomorphism over $\star(\sigma) = \sigma * \Link(\sigma)$ by induction over simplices of $\Link(\sigma)$.

For a simplex $\rho \subset \Link(\sigma)$ suppose we have a block smooth homeomorphism $\phi_{\sigma * \partial \rho} : p^{-1}(\sigma * \partial \rho) \to (\sigma * \partial \rho) \times M$. Then it does not necessarily agree with the restriction of the block chart $h_{\sigma * \rho}$, but differs from it by a block diffeomorphism of $(\sigma * \partial \rho) \times M$. The semisimplicial set $\widetilde{\Diff}_\bullet(M)$ is easily seen to be Kan, so this block diffeomorphism may be extended to a block diffeomorphism $\varphi$ of $(\sigma * \rho) \times M$, and then $\phi_{\sigma * \partial \rho}$ agrees with $\varphi \circ h_{\sigma * \rho}\vert_{\sigma * \partial \rho}$. Thus we may extend $\phi_{\sigma * \partial \rho}$ to a block smooth homeomorphism over $\sigma * \rho$.
\end{proof}

\section{Block bundles have MMM-classes}\label{blockbundles-mmm}

In order to show that block bundles admit generalised MMM-classes, we will prove more specifically that to a block bundle $p:E \to \vert K \vert$ over a finite simplicial complex we can associate the following structures, naturally in the block bundle:

\begin{enumerate}[(i)]
\item A Leray--Serre spectral sequence.
\item A transfer map $\trf_p^*:H^*(E) \to H^*(\vert K \vert)$ of Becker--Gottlieb type (not at all to be confused with the Gysin map).
\item A \emph{stable} vertical tangent bundle $T_v^s E \to E$.
\end{enumerate}
If the block bundle is an actual fibre bundle, then all these structures will reduce to those coming from the smooth bundle structure. Once this work is done, Theorem \ref{thm:A} in the case where the base is a finite simplicial complex is proved by the following line of argument.

\begin{proof}[Proof of Theorem \ref{thm:A} for finite simplicial complexes]
First note that the Gysin map $p_!$ can be defined in terms of the Leray--Serre spectral sequence, see \cite[\S 8]{BH}, as long as the fibres are compatibly oriented. Recall that if $\cha(\bF) \neq 2$ we let $c$ be a monomial in $\bF[p_1, p_2, \ldots]$ if $d$ is odd, and a monomial in $\bF[p_1, p_2, \ldots] \langle 1, e \rangle$ if $d$ is even, and if $\cha(\bF)=2$ then we let $c$ be a monomial in $\bF[w_1, w_2, \ldots]$. We aim to define $\tilde{\kappa}_c(p, \mathcal{A})$.

If $\cha(\bF)=2$ or if $d$ is odd (so $e=0$), the monomial is expressed in terms of just Pontrjagin classes and Stiefel--Whitney classes. These are \emph{stable}, in the sense that they only depend of the stable isomorphism class of a vector bundle, so we may define
$$\tilde{\kappa}_c(p, \mathcal{A}) := p_!(c(T_v^s E)).$$

The non-stable characteristic classes with field coefficients appear for $\cha(\bF) \neq 2$ and $d$ even, and are those of the form $e \cdot q (p_1,\ldots,p_n)$ for some polynomial $q$ in the Pontrjagin classes. For a smooth bundle $\pi:E \to B$ and a class $x \in H^* (E)$, the identity $\trf_{\pi}^{*} (x)=\pi_! (e (T_v E) \cdot x)$ holds (see \cite{BG75}): for block bundles, we use this formula as a \emph{definition}: 
$$\tilde{\kappa}_{e \cdot q(p_1, \ldots, p_n)}(p, \mathcal{A}) :=\trf_{p}^{*} (q (p_1(T_v^s E),\ldots,p_n(T_v^s E))).$$

As we have discussed, on smooth fibre bundles these definitions recover the usual $\kappa_c$, and they are also natural because the structures i) -- iii) used in their definition are natural. This finishes the proof of Theorem \ref{thm:A}.
\end{proof}

We still have to construct the three structures listed above. The key to the transfer and the Leray--Serre spectral sequence is Proposition \ref{block-bundles:quasifibration}, and does not require that $K$ is finite. If $(p : E \to \vert K \vert, \mathcal{A})$ is an block bundle with fibre over a vertex $v$ the oriented $d$-manifold $M$, the inclusion
$$M = p^{-1}(v) \lra \hofib_p (v)$$
is a weak homotopy equivalence by Proposition \ref{block-bundles:quasifibration}. Thus the Leray--Serre spectral sequence for the replacement $p^f : E^f \to \vert K \vert$ of $p$ by a fibration has the form
$$H^s(\vert K \vert ; \mathcal{H}^t(M)) \cong H^s(\vert K \vert ; \mathcal{H}^t(\hofib_p (v))) \Longrightarrow H^{s+t}(E^f) \cong H^{s+t}(E)$$
and is the desired spectral sequence. If $(p, \mathcal{A})$ is an oriented block bundle, then the local system $\mathcal{H}^d(\hofib_p (v))$ is trivialised, which allows us to define the pushforward using this spectral sequence.

A construction of the transfer that is sufficiently general for our purposes was given by Casson and Gottlieb \cite{CG}. The ``transfer theorem'' stated in the introduction of loc.\ cit.\ states that if $f:X \to Y$ is a Hurewicz fibration over a CW complex base whose fibres are homotopy equivalent to a finite CW complex then there is a transfer map $\trf_f^* : H^* (X) \to H^* (Y)$. We would like to apply this to $p : E \to \vert K \vert$, but this is not a fibration: if we na{\"i}vely replace it by one, we cannot expect its fibre to have the homotopy type of a CW complex. We instead prove  a general lemma showing that maps may be weakly replaced by Hurewicz fibrations with CW complex fibres.

\begin{lemma}\label{lemma:hurewicz-replacement}
Let $f_0:E_0 \to B_0$ be a map of spaces. Then there exists a Hurewicz fibration $f_3:E_3 \to B_3$ over a CW complex base which is weakly equivalent (via a zig-zag of maps) to $f_0$, such that the fibres of $f_3$ have the homotopy type of a CW complex.
\end{lemma}

\begin{proof}
We construct a commuative diagram
\begin{equation*}
\begin{gathered}
\xymatrix{
E_3 \ar[d]^{f_3} & \ar[l] E_2  \ar[d]^{f_2} \ar@{=}[r] & E_1 \ar[d]^{f_1} \ar[r] & E_0 \ar[d]^{f_0}\\
B_3\ar@{=}[r] & B_2 \ar[r]  & B_1  \ar[r] & B_0, 
}
\end{gathered}
\end{equation*}
such that all horizontal maps are weak homotopy equivalences and such that $f_3$ is a Hurewicz fibration with the desired property.
By taking the geometric realisations of the singular complexes of $B_0$ and $E_0$, we can replace $f_0$ by a cellular map $f_1: E_1 \to B_1$ of CW-complexes. The mapping cylinder of a cellular map is again a CW-complex, and so by replacing $f_1$ by the inclusion into its mapping cylinder, we turn $f_1$ into a cellular inclusion $f_2: E_2 \to B_2$. The homotopy fibre (using the standard construction, i.e. that in \cite[p. 59]{May}) of a cellular inclusion has the homotopy type of a CW complex by \cite[Thm. 3]{Miln}. Thus the fibration replacement $f_3: E_3 \to B_3$ of $f_2$ has the desired properties. 
\end{proof}

The construction of the transfer is deduced from this lemma as follows. Apply the lemma to $p : E \to \vert K \vert$ to obtain a weakly equivalent Hurewicz fibration $q : X \to Y$ over a CW complex base with CW complex fibres. The fibre of $q$ is weakly equivalent to the homotopy fibre of $p$, and so, by Proposition \ref{block-bundles:quasifibration}, to $M$: these both have the homotopy types of CW complexes, so they are actually homotopy equivalent, so $q$ has fibres homotopy equivalent to a \emph{finite} CW complex. We apply the ``transfer theorem'' of \cite{CG} to find a map
$$\trf_q^* : H^*(X) \lra H^*(Y),$$
which passing through the zig-zags of weak equivalences relating $q$ and $p$ gives the desired map
$$\trf_p^* : H^*(E) \lra H^*(\vert K \vert).$$

Finally, we have to produce a stable analogue of the vertical tangent bundle. If $\pi: E \to B$ is a smooth fibre bundle over a compact smooth manifold base, there exists an embedding $i : E \to B \times \bR^N$ over $B$ for some $N \gg 0$. Then we have bundle isomorphisms
$$TE \cong \pi^*TB \oplus T_v E \quad\quad TE \oplus \nu_i \cong \pi^*TB \oplus \epsilon^N \quad \quad \nu_i \oplus T_v E \cong \epsilon^N$$
where $\nu_i$ is the normal bundle of the embedding $i$. The last of these is the most convenient description to use to construct the stable vertical tangent bundle.

\begin{proposition}\label{existence:normal-bundle}
A smooth block bundle $(p : E \to \vert K \vert, \mathcal{A})$ over a finite simplicial complex has a natural stable bundle $T^s_vE \to E$, which when the block bundle arises from a smooth fibre bundle agrees with the vertical tangent bundle.
\end{proposition}

\begin{proof}
Let $(p : E \to \vert K \vert, \mathcal{A})$ be such a block bundle. We denote $E_{L}:=p^{-1} (\vert L \vert)$ for each subcomplex $L \subset K$. The proof begins with the construction of a suitable embedding.
As $K$ has finitely many simplices, we may find an embedding $e: E \hookrightarrow \vert K \vert \times \bR^N$ for large enough $N$, by induction on simplices, which is a smooth embedding on each block chart. We do not require that it is an embedding over $\vert K \vert$, but require that $e (p^{-1}(\sigma)) \subset \sigma \times \bR^N$ for each simplex $\sigma$.
Fix moreover an embedding $a:\vert K \vert \to \bR^k$ that is affine on each simplex.
We obtain an embedding of $E$ into $\bR^k \times \bR^N$. For each simplex $\sigma \subset K$ the linear embedding $\sigma \subset \vert K \vert \subset \bR^k$ induces a metric on $T\sigma$. We denote by $\mu_{\sigma}$ the normal bundle of $\sigma$ inside $\bR^k$, and by $\mu_{\tau}^{\sigma}$ the normal bundle of $\tau$ inside $\sigma$, when $\tau \subset \sigma$ is a subsimplex (these normal bundles are defined by taking orthogonal complements). We have a submanifold $p^{-1}(\sigma) \subset \sigma \times \bR^N$ of dimension $\dim (\sigma)+d$. We need to choose the embedding more carefully, though. Namely, we require that whenever $\tau \subset \sigma$ is a face, then the vector bundles $T p^{-1}(\sigma)|_{p^{-1}(\tau)}$ and $T p^{-1} (\tau) \oplus \mu_{\tau}^{\sigma}$ should agree. (As $Tp^{-1}(\tau)$ and $\mu_\tau^\sigma$ are transverse and the sum of their dimensions is the dimension of $Tp^{-1}(\sigma)\vert_{p^{-1}(\tau)}$, this condition is equivalent to asking $\mu_\tau^\sigma$ to lie in $Tp^{-1}(\sigma)\vert_{p^{-1}(\tau)}$.) This can be achieved by induction over skeleta: if $E_{K^{n}} \to \vert K \vert \times \bR^N$ is already constructed, we embed an open neighborhood of $E_{K^{n}} \subset E_{K^{n+1}}$ into $\vert K \vert \times \bR^{N+1}$ so that the desired property holds on this neighborhood and appeal to the relative version of the Whitney embedding theorem. We shall call such an embedding \emph{good} and fix a good embedding.

We have a submanifold $p^{-1}(\sigma) \subset \sigma \times \bR^N$ of dimension $d+ \dim(\sigma)$ which at each point $x \in p^{-1}(\sigma)$ has a $(\dim(\sigma)+d)$-dimensional tangent subspace $T_x(p^{-1}(\sigma)) \subset T_x(\sigma \times \bR^N)$. By taking $T_x(p^{-1}(\sigma)) \oplus \mu_{\sigma}$, we obtain a $(d +k)$-dimensional vector space, hence a point in $\Gr_{d+k}(\bR^{k+N})$. The resulting map $t_{\sigma,e,a}: p^{-1} (\sigma) \to \Gr_{d+k}(\bR^{k+N})$ is continuous; and by the property of a good embedding, we have that $t_{p^{-1}(\sigma),a,e}|_{p^{-1}(\tau)} = t_{p^{-1}(\tau),a,e}$, so these glue together to a continuous map $t_{E,a,e}:E \to \Gr_{d+k}(\bR^{N+k})$. The bundle $t^{*}_{E,a,e}\gamma_{k+d}$ on $E$ is a $k+d$-dimensional subbundle of $E \times \bR^{N+k}$, which we call the \emph{stable vertical tangent bundle} of $E$ with respect to the embeddings $e$ and $a$. We allow ourselves to denote the bundle by the same symbol, namely $t_{E,a,e}$. In the same way, we can, for $x \in p^{-1}(\sigma)$, take $n_{\sigma,a,e} (x)$ to be the orthogonal complement of $t_{\sigma} (x)$ in $\bR^{N+k}$, this is, in each block, the $(N-d)$-dimensional normal bundle of $p^{-1}(\sigma) \subset \sigma \times \bR^N$. The resulting $(N-d)$-dimensional vector bundle on $E$ will be denoted $n_{E,a,e}$. By construction, it is clear that the following properties hold.

\begin{enumerate}[(i)]
\item If $L \subset K$ is a subcomplex, then $n_{E,a,e}|_{E_L}=n_{E_L,a|_L,e|_{E_L}}$.
\item If $i:\bR^k \to \bR^{k+l}$ is an affine isometric embedding, then the normal bundle is unchanged, and to the stable vertical tangent bundle a trivial bundle is added.
\item If $j: \bR^N \to \bR^{N+l}$ is an affine isometric embedding, the stable vertical tangent bundle is unchanged, while to the normal bundle a trivial bundle is added.
\item $n_{E,a,e} \oplus t_{E,a,e} = \epsilon^{N+k}$.
\end{enumerate}

Now we claim that, for fixed $a$, the stable isomorphism class of $n_{E,a,e}$ does not depend on $e$. Let $e_0$ and $e_1$ be two good embeddings. By the third property, we can assume that the dimension of the target is the same for both embeddings. Choose a concordance $(p' : E' \to \vert K \vert \times [0,1], \mathcal{A}')$ from $(p : E \to \vert K \vert, \mathcal{A})$ to itself (cf.\ paragraph after Definition \ref{defn:concordance}). The embeddings $e_0$ and $e_1$ together give a good embedding $e:E'\vert_{\vert K \vert \times \{0,1\}} = E \times \{0,1\} \to (\vert K \vert \times \{0,1\}) \times \bR^N$ and, if $N$ is large enough, we can extend this to a good embedding $e:E' \to \vert K \vert \times I \times \bR^N$. Take the affine embedding $a \times \id_I: \vert K \vert \times I \to \bR^{k+1}$. By the first and second property, we find that $n_{E',a\times \id_I,e}|_{E \times \{i\}} \cong n_{E,a,e_i}$, for $i=0,1$, and so by homotopy invariance of vector bundles this proves the claim. There is no need for us to prove the independence of $a$, as there is a canonical affine embedding: suppose that $K$ has $k$ vertices, and take the canonical embedding $\vert K \vert \to \bR^k$. Now we define the stable vertical tangent bundle of $E$ to be 
$$T_{v}^{s} E := t_{E,a,e} - \epsilon^k$$
for some good embedding. We have proved that its stable isomorphism class does not depend on the choices we made.

Suppose that the block bundle arises from a smooth fibre bundle $\pi : E \to \vert K \vert$, with vertical tangent bundle $T_v E \to E$. We fix the canonical affine embedding from before and can pick a good embedding, and this time there is no problem to define the embedding to be over $\vert K \vert$. On each block chart, the map $\pi: E_{\sigma} \to \sigma$ is a submersion, and the kernel of its differential is equal to $T_v E|_{E_{\sigma}}$. Taking direct sum with the identity on $\mu_{\sigma}$, we obtain a bundle epimorphism 
$$\eta_{\sigma}:t_{\sigma,a,e}  \lra T \sigma \oplus \mu_{\sigma} = \bR^k$$
and if $\tau \subset \sigma$ is a face, then $\eta_{\sigma}|_{\tau} = \eta_{\tau}$, because we chose the embedding to be good. Altogether, we obtain a short exact sequence
$$0 \lra T_v E \lra t_{E,a,e} \lra \epsilon^k \lra 0,$$
which proves that the stable vector bundles $T_v E$ and $T_{v}^{s} E$ are isomorphic.
\end{proof}

This finishes the construction of the data required to define MMM-classes for smooth block bundles over finite simplicial complexes. The following lemma is not necessary for the construction, but will be useful in Section \ref{sec:counterexample}.

\begin{lemma}\label{lem:BlockBundleNormalManifold}
Let $p : E \to B$ be a smooth map between smooth manifolds, $\phi : \vert K \vert \to B$ be a Whitehead triangulation, and $\mathcal{A}$ be a block bundle structure on $\phi^*p : \phi^*E \to \vert K \vert$ with fibre $M$. Then under the homeomorphism $\hat{\phi} : \phi^*E \approx E$ induced by $\phi$ there is a stable isomorphism $T_v^s E \cong_s \hat{\phi}^*(TE - p^*TB)$.
\end{lemma}
\begin{proof}
Choose a smooth embedding $e : E \to B \times \bR^N$ over $B$, which has a normal bundle $\nu(e)$, and note that there is an isomorphism $TE \oplus \nu(e) \cong p^*TB \oplus \epsilon^N$. Let us write $\phi^*e : \phi^*E \hookrightarrow \vert K \vert \times \bR^N$ for the induced embedding. We can pick $e$ so that $\phi^* e$ is a good embedding (in the sense of the proof of Proposition \ref{existence:normal-bundle}).
What we have to show is that $\hat{\phi}^{*}\nu(e) \cap t_{\phi^* E,a,\phi^* e} =0$ as subbundles of $\epsilon^{N+k}$; as they are subbundles of  complementary dimension, it will follow that $\hat{\phi}^{*}\nu(e) \oplus t_{\phi^* E,a,\phi^* e} \cong \epsilon^{N+k}$, which gives the required stable isomorphism. Since the embedding is good, the restriction of $\nu(e)$ to $E_{\sigma}$ is the same as the normal bundle of $E_{\sigma}$ inside $\sigma \times \bR^N$, which proves that the two subbundles are transverse as required.
\end{proof}

We will now explain how to extend the definition of MMM-classes to arbitrary base simplicial complexes, by defining universal MMM-classes. 

\begin{theorem}\label{thm:Aprime}
For any $d$-manifold $M$ and any monomial $c$ as in Theorem \ref{thm:A} there is a class $\tilde{\kappa}_c \in H^*(B\blockdiff(M);\bF)$ satisfying
\begin{enumerate}[(i)]
\item for a map $f: \vert K \vert \to B\blockdiff(M)$ from a finite simplicial complex classifying a block bundle $(p : E \to \vert K \vert, \mathcal{A})$, we have $f^*\tilde{\kappa}_c = \tilde{\kappa}_c(p, \mathcal{A})$,
\item under the natural map $B\Diff(M) \to B\blockdiff(M)$ the class $\tilde{\kappa}_c$ pulls back to $\kappa_c$. 
\end{enumerate}
\end{theorem}

\begin{proof}

We have given a homotopy equivalence $B\blockdiff(M) := \vert N_\bullet \blockdiff(M)_\bullet \vert \simeq \vert\mathcal{M}(M)_\bullet \vert$. By an observation of Rourke and Sanderson \cite[p.\ 327]{RSI}, the second derived subdivision of a semi-simplicial set has the structure of a simplicial complex, so there is a homotopy equivalence $\vert\mathcal{M}(M)_\bullet \vert \simeq \vert L \vert_w$ for some (infinite) simplicial complex $L$, where $\vert - \vert_w$ denotes the geometric realisation with the weak topology. 

For each finite sub-simplicial complex $K \subset L$ the map $\vert K \vert \to \vert L \vert_w \simeq \vert\mathcal{M}(M)_\bullet \vert$ classifies a block bundle $(p_K : E_K \to \vert K \vert, \mathcal{A}_K)$, unique up to concordance, and we have defined $\tilde{\kappa}_c(p_K, \mathcal{A}_K) \in H^*(\vert K\vert ;\bF)$. If $K' \subset K$ is a subcomplex, then $(p_{K'} : E_{K'} \to \vert K' \vert, \mathcal{A}_{K'})$ is concordant to $(p_K\vert_{K'} : E_K\vert_{K'} \to \vert K' \vert, \mathcal{A}_K\vert_{K'})$, as both are classified by the same homotopy class of map to $\vert\mathcal{M}(M)_\bullet \vert$, and so $\tilde{\kappa}_c(p_{K'}, \mathcal{A}_{K'})$ is equal to the restriction of $\tilde{\kappa}_c(p_K, \mathcal{A}_K)$. Thus we obtain a class
$$\tilde{\kappa}_c \in \lim_{K \subset L}H^*(\vert K\vert ;\bF) \cong H^*(\cup_{K \subset L} \vert K \vert ;\bF) \cong H^*(\vert L \vert_w ;\bF) \cong H^*(B\blockdiff(M);\bF).$$
The first isomorphism holds as for each $i$ the inverse system $\{H^i(\vert K\vert ;\bF)\}_{K \subset L}$ consists of finite-dimensional vector spaces, so is Mittag-Leffler and has no $\lim^1$. The second isomorphism holds as $\vert L \vert_w$ has the weak topology so is the colimit of its finite subcomplexes. This class enjoys the properties claimed.
\end{proof}

\section{Topological bundles have MMM-classes}\label{topbundles}

Let $\pi: E \to B$ be an oriented bundle of closed oriented $d$-dimensional topological manifolds over a compact topological manifold base. The data $\ft_v E:= (E \overset{\Delta}\to E \times_B E \overset{\pi_1}\to E)$ describes an oriented $d$-dimensional topological microbundle \cite{MilnMicro} over $E$, and as $E$ is a manifold bundle over a compact base it is again (para)compact and so the microbundle $\ft_v E$ is representable by an oriented $\bR^d$-bundle $T_v E$, by the Kister--Mazur theorem \cite[Theorem 2]{Kister}.

Now, an oriented $\bR^d$-bundle $V \to X$ has an Euler class and Stiefel--Whitney classes, but also has \emph{rational} Pontrjagin classes.
Euler and Stiefel--Whitney classes are invariants of the underlying spherical fibration $V \setminus 0 \to X$: for example, the total Stiefel--Whitney class is defined as $\thom^{-1} (\Sq \thom (1)) \in H^{*}(X; \bF_2)$, where $\thom: H^* (X;\bF_2) \overset{\sim}\to H^{*+d}(V,V\setminus 0;\bF_2)$ is the Thom isomorphism. The existence of rational Pontrjagin classes for $\bR^d$-bundles is much deeper and goes back to Novikov's theorem on topological invariance of rational Pontrjagin classes, \cite{Nov}. One way to view these rational Pontrjagin classes is the fact that $TOP/O$ has finite homotopy groups \cite[Ess.\ V Thm.\ 5.5]{KS}, so $BO \to BTOP$ is a rational homotopy equivalence.

\begin{remark}
If $V \to X$ is an oriented vector bundle of rank $d$, then $p_m =0$ if $4m > 2d$ and $p_m =e^2$ if $d=2m$. The question of whether these identities hold for the rational Pontrjagin classes of topological $\bR^d$-bundles is a difficult open problem, cf.\ \cite{ReisWeiss}.
\end{remark}

Thus we may define, for $c \in \bF[e,p_1,p_2,\ldots ]$ (if $\cha(\bF) = 0$, and ignoring $e$ if $d$ is odd) or $c \in \bF [w_1,w_2,\ldots,w_d]$ (if $\cha(\bF) = 2$) the class
$$\kappa^{TOP}_c(\pi) := \pi_!(c(T_v E)) \in H^*(B;\bF).$$
The classes so defined are clearly natural under pull-back, and agree with the $\kappa_c$ for smooth bundles. This provides the construction for Theorem \ref{thm:B} as long as the bundle in question has a compact topological manifold base. 

For a closed oriented topological manifold $M$, we let $\Homeo^+(M)$ denote the group of orientation-preserving homeomorphisms of $M$, in the compact-open topology, and let $B\Homeo^+(M)$ be its classifying space. It carries a universal fibre bundle
$$\pi: E := E\Homeo^+(M) \times_{\Homeo^+(M)} M \lra B\Homeo^+(M).$$
Theorem \ref{thm:B} is immediate from the following proposition.

\begin{proposition}\label{prop4.2}
Let $\cha(\bF)=0$ or $2$. There exist unique classes $\kappa_c^{TOP} \in H^*(B\Homeo^+(M);\bF)$ which pull-back to the classes $\kappa_c^{TOP}(\pi)$ for every oriented bundle $\pi : E \to B$ with fibre $M$ over a compact manifold.
\end{proposition}
\begin{proof}
Consider first the case $\cha(\bF)=0$.
Let $c$ have degree $k$, so $\kappa_c^{TOP}$ should have degree $(k-d)$. Let $f : B^{k-d} \to B\Homeo^+(M)$ be a continuous map from an $(k-d)$-dimensional smooth oriented manifold. This classifies a fibre bundle $\pi : E \to B$ over a compact manifold base, and we may extract a rational number $\int_B \kappa_c^{TOP}(\pi)$. The usual argument shows that this number is invariant if we change the map $f$ by a cobordism, so we obtain a linear map
$$\int_{-} \kappa_c^{TOP}: \Omega_{k-d}^{SO}(B\Homeo^+(M)) \lra \bF$$
from the oriented bordism of $B\Homeo^+(M)$.

Furthermore, if $g: B^{k-d-\ell} \to B\Homeo^+(M)$ is a continuous map classifying a bundle $\pi : E \to B$ and $N^\ell$ is a $\ell$-dimensional manifold, then the bundle $\mathrm{Id}_N \times \pi: N \times E \to N \times B$ is pulled back from the projection to $B$, so
$$\int_{N \times B} \kappa_c^{TOP}(\mathrm{Id}_N \times \pi) = \begin{cases}
0 & \ell > 0\\
[N]\cdot \int_B \kappa_c^{TOP}(\pi) & \ell=0.
\end{cases}$$
Thus $\int_{-} \kappa_c^{TOP}$ descends to a map
$$\int_{-} \kappa_c^{TOP} : \Omega^{SO}_*(B\Homeo^+(M))\otimes_{\Omega^{SO}_*(*)} \bF \cong H_*(B\Homeo^+(M);\bF) \lra \bF[k-d],$$
so represents a class ${\kappa}_c^{TOP} \in H^{k-d}(B\Homeo^+(M);\bF)$. For an oriented bundle $\pi : E \to B$ classified by a map $f : B \to B\Homeo^+(M)$ we have $f^*\kappa_c^{TOP} = \kappa_c^{TOP}(\pi)$, as both classes give the same function $\Omega_{k-d}^{SO}(B) \to \bF$.

The case $\cha(\bF)=2$ is the same, but replacing oriented bordism $\Omega^{SO}_*(-)$ by unoriented bordism $\Omega_*(-)$, and using the fact that $\Omega_*(-) \otimes_{\Omega_*(*)} \bF \cong H_*(-;\bF)$.
\end{proof}

\section{Proof of Theorem \ref{thm:counterexample}}\label{sec:counterexample}

The techniques used in this proof we inspired by \cite[\S 5]{HSS}. We will aim to find a homotopy equivalence $f : E \overset{\sim}\to S^{12} \times \bH \bP^2$ with 
\begin{enumerate}[(i)]
\item $p_1(TE) = f^*p_1(S^{12} \times \bH \bP^2)$,

\item $p_2(TE) = f^*p_2(S^{12} \times \bH \bP^2)$,

\item  but $p_5(TE) \neq 0$.
\end{enumerate}
Supposing we have done so, we try to give $p:=\pi_1 \circ f : E \to S^{12}$ the structure of a smooth block bundle, using the work of Casson \cite{Casson}. In particular his Theorem 1 applies to $p$, and gives a single obstruction which in our case may be described as follows. We may homotope $f$ to be smooth and transverse to $\{b\} \times \bH \bP^2$ for some $b \in S^{12}$, giving a pull-back square
\begin{equation*}
\xymatrix{
F^8 \ar[r]^-g \ar@{^(->}[d]& \{b\} \times \bH \bP^2 \ar@{^(->}[d]\\
E^{20}  \ar[r]^-f_-\simeq & S^{12} \times \bH \bP^2.
}
\end{equation*}
The map $g$ is a degree one normal map, as $f$ is and the two vertical maps are embeddings with trivialised normal bundle. Casson's obstruction is then the surgery obstruction for $g$, i.e.\ $\tfrac{1}{8}(\sign(F) - \sign(\bH\bP^2))$. As the vertical embeddings are normally framed, we may compute
$$\sign(F) = \langle \cl_2(TF), [F]\rangle = \langle \cl_2(TE), [F]\rangle = \langle f^*\cl_2(T(S^{12} \times \bH \bP^2)), [F]\rangle = \sign(\bH \bP^2),$$
so Casson's obstruction vanishes and $p$ is homotopic to a ``prefibration". By \cite[Lemma 6]{Casson} any prefibration is equivalent to a smooth block bundle, $p: E \to \vert K \vert \cong S^{12}$. By Lemma \ref{lem:BlockBundleNormalManifold} we have $T_v^s E \simeq_s TE - p^*TS^{12}$, and so $p_5(T_v^s E) = p_5(TE) \neq 0$. Thus $\int_{S^{12}} \tilde{\kappa}_{p_5} = \int_E p_5(TE) \neq 0$, which finishes the proof of Theorem \ref{thm:counterexample}.

It remains to produce the homotopy equivalence $f : E \overset{\sim}\to S^{12} \times \bH \bP^2$ with the properties claimed above. We do so by surgery theory, using a result which is neatly packaged in \cite[Theorem 6.5]{Davis}. Namely, if we write $x \in H^{12}(S^{12};\bQ)$ and $y \in H^{4}(\bH \bP^2;\bQ)$ for generators, then by the cited theorem there exists a manifold $E$ and homotopy equivalence $f$ such that
$$\cl(TE) = f^*\big(\cl(T(S^{12} \times \bH \bP^2)) + R \cdot x \cdot y\big)$$
for some non-zero integer $R$. As $x \cdot y$ has degree 16, the first two of the desired properties hold. To establish the last desired property, we simply compute with the Hirzebruch $\cl$-polynomials \cite[p.\ 12]{Hir}. First note that
$$p(T(S^{12} \times \bH \bP^2)) = 1 + 2y + 7 y^2.$$
Now, using
$$\cl_4 = \tfrac{1}{3^4 \cdot 5^2 \cdot 7}\big(381 p_4 - 71 p_2 p_1 - 19 p_2^2 + 22 p_2 p_1^2 - 3 p_1^4\big)$$
and noting that $TE$ and $T(S^{12} \times \bH \bP^2)$ have the same Pontrjagin classes below the fourth, we obtain
$$\tfrac{381}{3^4 \cdot 5^2 \cdot 7}\cdot p_4(TE)  = f^*\left(\tfrac{381}{3^4 \cdot 5^2 \cdot 7} \cdot p_4(T(S^{12} \times \bH \bP^2)) + R \cdot x \cdot y\right)$$
and so
$$p_4(TE) = f^*\big(\tfrac{3^4 \cdot 5^2 \cdot 7}{381} \cdot R \cdot x \cdot y \big).$$
Secondly, using
$$\cl_5 = \tfrac{1}{3^5 \cdot 5^2 \cdot 7\cdot 11}\big(5110 p_5 -919 p_4 p_1 - 336 p_3 p_2 + 237 p_3 p_1^2 + 127 p_2^2 p_1 - 83p_2 p_1^3 + 10 p_1^5\big)$$
and the fact that $TE$ and $T(S^{12} \times \bH \bP^2)$ have the same Pontrjagin classes below the fourth, we obtain
$$5110 \cdot(p_5(TE) - f^*p_5(T(S^{12} \times \bH \bP^2))) = 919 \cdot(p_4(TE) - f^*p_4(T(S^{12} \times \bH \bP^2))) \cdot p_1(TE)$$
and so
$$p_5(TE) = \tfrac{919}{5110} \cdot f^*\big(\tfrac{3^4 \cdot 5^2 \cdot 7}{381} \cdot R \cdot x \cdot y\big) \cdot f^*\big(2 \cdot y\big) = f^*\big(\tfrac{124065}{9271} \cdot R \cdot x \cdot y^2\big) \neq 0.$$

\section{Proof of Theorem \ref{thm:app}}

The main theorems of \cite{GRW1, GRW2} imply that $H^*(B\Diff(W_{g}^{2n}, D^{2n});\bQ)$ is generated by generalised MMM-classes in degrees $* \leq \frac{g-4}{2}$. As these classes may be defined on $B{\Homeo}(W_{g}^{2n}, D^{2n})$ by the results of Section \ref{topbundles} of the present paper, we immediately find that
$$H^*(B{\Homeo}(W_{g}^{2n}, D^{2n});\bQ) \lra H^*(B\Diff(W_{g}^{2n}, D^{2n});\bQ)$$
is surjective in degrees $* \leq \frac{g-4}{2}$. The same argument, using Section \ref{blockbundles-mmm} instead, proves the surjectivity for the comparison map between diffeomorphisms and block diffeomorphisms. In \cite[Theorem 5.1]{ERW}, we proved (or rather derived from results by Waldhausen, Igusa, Farrell--Hsiang and others) that $B\Diff(W_{g}^{2n}, D^{2n}) \to B\blockdiff(W_{g}^{2n}, D^{2n})$ induces an isomorphism in rational cohomology, in degrees $* \leq \min\big(\tfrac{2n-7}{2}, \tfrac{2n-4}{3}\big)$. Thus it is left to prove the next proposition.

\begin{proposition}
The map
$$H^*(B\Homeo(W_g^{2n}, D^{2n});\bQ) \lra H^*(B\Diff(W_g^{2n}, D^{2n});\bQ)$$
is a split injection in degrees $* \leq \min\big(\tfrac{2n-7}{2}, \tfrac{2n-4}{3}\big)$.
\end{proposition}
\begin{proof}
Note that the statement only has content for $2n \geq 10$, so we may as well suppose this is the case. Let us denote by $F$ the homotopy fibre of
$$B\Diff(W_g^{2n}, D^{2n}) \lra B\Homeo(W_g^{2n}, D^{2n}).$$
We make two claims: that $\pi_0(F)$ is a finite set, and that each path-component of $F$ has trivial rational homology in degrees $* \leq \min\big(\tfrac{2n-7}{2}, \tfrac{2n-4}{3}\big)$. Granted these claims, the Leray--Serre spectral sequence for $f$ is supported along the line $q=0$ in total degrees $p+q \leq \min\big(\tfrac{2n-7}{2}, \tfrac{2n-4}{3}\big)$, so there is an isomorphism
$$H^*(B\Homeo(W_g^{2n}, D^{2n});\bQ[\pi_0(F)]) \cong H^*(B\Diff(W_g^{2n}, D^{2n});\bQ)$$
in degrees $* \leq \min\big(\tfrac{2n-7}{2}, \tfrac{2n-4}{3}\big)$. The proposition now follows from the maps of coefficient systems
$$\bQ \overset{1 \mapsto \sum x}\lra \bQ[\pi_0(F)] \overset{\epsilon}\lra \bQ.$$

It remains to prove the two claims. As we have supposed that $2n \geq 10$, smoothing theory (cf.\ \cite[Ess.\ V \S 3]{KS}) applies, and provides a map from $F = \Homeo(W_g^{2n}, D^{2n}) / \Diff(W_g^{2n}, D^{2n})$ to the space $\Gamma(W_g^{2n}, D^{2n})$ of lifts in the diagram
\begin{equation*}
\xymatrix{
D^{2n} \ar[r]\ar[d] & BO(2n) \ar[d]\\
W_g^{2n} \ar[r]\ar@{-->}[ru] & BTOP(2n)
}
\end{equation*}
and shows that $F \to \Gamma(W_g^{2n}, D^{2n})$ is a homotopy equivalence onto those path components which it hits. Thus it is enough to show that $\pi_0(\Gamma(W_g^{2n}, D^{2n}))$ is finite and that each path-component of $\Gamma(W_g^{2n}, D^{2n})$ has trivial rational homology in degrees $* \leq \min\big(\tfrac{2n-7}{2}, \tfrac{2n-4}{3}\big)$.

Choose a handle decomposition of $W_g^{2n}$ with a single 0-handle the disc $D^{2n}$, $2g$ $n$-handles, and a single $2n$-handle. Let us write $M$ for the union of the handles of index less than $2n$, and $\Gamma(M, D^{2n})$ for the analogous space of lifts for $M$. The tangent bundle of $W_g^{2n}$ is trivial when restricted to $M$, so choosing a trivialisation gives an equivalence
$$\Gamma(M,D^{2n}) \simeq \left[\Omega^n \left(\frac{TOP(2n)}{O(2n)}\right)\right]^{2g}.$$
Recall that $TOP(2n) / O(2n) \to TOP / O$ is $(2n+1)$-connected \cite[Ess.\ V, Thm.\ 5.2]{KS}, and ${TOP / O}$ has finite homotopy groups (they are the groups of exotic spheres), so the set of path components of $\Gamma(M,D^{2n})$ is a finite set, and each path component has trivial rational homology in degrees $* \leq 2n-n = n$.

Restricting lifts gives a fibration $\Gamma(W_g^{2n},D^{2n}) \to \Gamma(M,D^{2n})$, and the fibre over a point is either empty, or is homotopy equivalent to $\Omega^{2n}\left({TOP(2n) / O(2n)}\right)$. It will be enough to show that this space has finitely-many path components and trivial rational homology in positive degrees (in a range). As $\Homeo(D^{2n}, \partial D^{2n})$ is contractible (by the Alexander trick), we have
$$\Homeo(D^{2n}, \partial D^{2n}) / \Diff(D^{2n}, \partial D^{2n}) \simeq B\Diff(D^{2n}, \partial D^{2n})$$
and smoothing theory again provides a map
$$\Homeo(D^{2n}, \partial D^{2n}) / \Diff(D^{2n}, \partial D^{2n}) \lra \Omega^{2n}\left(\frac{TOP(2n)}{O(2n)}\right)$$
which is a homotopy equivalence onto the path component which it hits. The set $\pi_0(\Omega^{2n}({TOP(2n) / O(2n)})) = \pi_{2n}(TOP(2n) / O(2n))$ is finite as above, and a theorem of Farrell and Hsiang \cite{FHs} shows that the rational homotopy groups $\pi_k(B\Diff(D^{2n}, \partial D^{2n}))\otimes \bQ$ are zero for $k \leq \min\big(\tfrac{2n-7}{2}, \tfrac{2n-4}{3}\big)$ (see \cite[\S 6.1]{WW2} for a careful treatment of the range). Thus $\Omega^{2n}_0({TOP(2n) / O(2n)})$ has trivial rational homology in this range.
\end{proof}

\bibliographystyle{alpha}

\end{document}